\documentclass[12pt]{amsart}

\usepackage[utf8]{inputenc}
\usepackage[T1]{fontenc}

\usepackage[dvipsnames]{xcolor}
\usepackage{amsmath, amsfonts, amsthm, amssymb}
\usepackage[top=1in, bottom=1in, left=1in, right=1in]{geometry}
\usepackage[backref=page, colorlinks=true, allcolors=blue]{hyperref}
\usepackage[alphabetic,abbrev,lite,backrefs]{amsrefs} 
\usepackage{tikz}
\usepackage{changepage}
\usepackage{mathrsfs}
\usepackage{tikz-cd}
\usepackage{mathtools} 
\usepackage{bm} 
\usepackage{caption}
\usepackage{blkarray} 
\usepackage{enumitem} 
\usepackage{algorithm}
\usepackage{algpseudocode}
\usepackage{multirow}

\usepackage[colorinlistoftodos,prependcaption,textsize=scriptsize]{todonotes}
\newcommand{\indep}{\perp \!\!\! \perp}

\makeatletter
\providecommand\@dotsep{5}
\renewcommand{\listoftodos}[1][\@todonotes@todolistname]{%
  \@starttoc{tdo}{#1}}
\makeatother

\theoremstyle{plain}
	\newtheorem{theorem}{Theorem}[section]
    \newtheorem*{theorem*}{Theorem}
	\newtheorem{lemma}[theorem]{Lemma}
    \newtheorem{corollary}[theorem]{Corollary}

    \newtheorem*{Question*}{Question}
\theoremstyle{definition}
    \newtheorem{defn}[theorem]{Definition}
    \newtheorem*{defn*}{Definition}
    \newtheorem{example}[theorem]{Example}
    \newtheorem*{example*}{Example}
    
\theoremstyle{remark}
	
	\newtheorem*{remark*}{Remark}

\AtBeginEnvironment{example}{%
  \pushQED{\qed}%
}
\AtEndEnvironment{example}{\popQED\endexample}

\usepackage{mathtools}
\newcommand\SetSymbol[1][]{\nonscript\:#1\vert\allowbreak\nonscript\:\mathopen{}}
\providecommand\given{} 
\DeclarePairedDelimiterX\Set[1]\{\}{\renewcommand\given{\SetSymbol[\delimsize]}#1}

\newcommand{\mc}[1]{\mathcal{#1}}

\newcommand{\res}[2][]{\left.{#2}\right|_{#1}}

\def\R{{\mathbb{R}}}
\def\N{{\mathbb{N}}}
\def\P{{\mathbb{P}}}
\def\C{{\mathbb{C}}}
\def\H{{\mc{H}}}

\def\M{{\mc{M}}}

\def\N{{\mathbb N}}

\DeclareMathOperator{\rank}{rank}

\DeclareMathOperator{\dlog}{dlog}
\DeclareMathOperator{\codim}{codim}

\title{Likelihood Correspondence of Toric Statistical Models}

\author{David Barnhill}
\address{Department of Mathematics, United States Naval Academy, Annapolis, MD}
\email{barnhill@usna.edu}

\author{John Cobb}
\address{Department of Mathematics, Auburn University, Auburn, AL}
\email{jdcobb3@gmail.com}

\author{Matthew Faust}
\address{Department of Mathematics, Michigan State University, East Lansing, MI}
\email{mfaust@msu.edu}

\begin{document}

\maketitle
\vspace{-0.3in}
\begin{abstract}
    Maximum likelihood estimation (MLE) is a fundamental problem in statistics. Characteristics of the MLE problem for discrete algebraic statistical models are reflected in the geometry of the \textit{likelihood correspondence}, a variety that ties together data and their maximum likelihood estimators. We construct this ideal for the large class of toric models and find a Gr\"{o}bner basis in the case of complete and joint independence models arising from multi-way contingency tables. All of our constructions are implemented in \textit{Macaulay2} in a package \texttt{LikelihoodGeometry} along with other tools of use in algebraic statistics. We end with an experimental section using these implementations on several interesting examples.
\end{abstract}


\section{Introduction}

Consider a process resulting in $n+1$ outcomes, each with probability $p_i$ of occurring. Running this process some number of times, we obtain a data vector $u=(u_0, \dots, u_n) \in \N^{n+1}$ where $u_i$ counts the number of trials resulting in outcome $i$. A \textit{discrete statistical model} is a variety $\M$ restricted to the probability simplex $\Delta_n = \Set*{p = (p_0,\dots, p_n) \in \R^{n+1}_{\geq 0} \given \sum_{i=0}^n p_i = 1}$ where the probability of the $i$-th state occurring is given by $p_i$. One method of estimating the set of parameters $p$ most likely to produce the observed $u$ is known as {\em maximum likelihood estimation}.  A maximum likelihood estimate for the parameter $p$ associated with $\M\subseteq\Delta_n$ is found using the likelihood function $\ell_u(p) = \prod_{i=0}^n p_i^{u_i}$. Maximum likelihood estimation (MLE) is a fundamental computational problem in statistics. Research into maximum likelihood estimation has yielded significant innovation in recent years with insights stemming from studying the geometry of algebraic statistical models as varieties \cites{Catanese2004TheML, huh2013maximum,  LikelihoodGeometry,Amendola2019MaximumLE,Maxim2022LogarithmicCB}. 

We make any questions about semi-algebraic sets easier by disregarding inequalities, replacing real numbers with complex numbers, and replacing affine space with projective space. To this end, fix a system of homogeneous coordinates $p_0, p_1, \dots, p_n$ of complex projective space $\P_p^n$ representing the space of probabilities. Similarly, we consider $u$ as coordinates of the data space $\P^n_u$. The likelihood function on $\P_p^n$ is given by
\begin{equation*}
    \ell_u(p) = \frac{p_0^{u_0}p_1^{u_1}\cdots p_n^{u_n}}{(p_0+p_1+\cdots + p_n)^{u_0+u_1+\cdots+u_n}}.
\end{equation*}
We aim to study the MLE problem on the restriction of $\ell_u$ to a given closed irreducible subvariety $\M\subseteq \P_p^n$. Fixing a data vector $u$, maximum likelihood estimates are given by the maxima of $\ell_u$. These maxima belong to the critical set of the log--likelihood function $\log \ell_u$. This critical set is given by $V(\dlog \ell_u) \subset \P^n_p$, where $\dlog \ell_u$ is the gradient of $\log \ell_u$ and $V(\dlog \ell_u)$ denotes its vanishing set. By varying $u$, we can collect such critical sets into a universal family $\mathcal{L}_\M \subset \P_p^n \times \P_u^n$ called the \textit{likelihood correspondence} that encodes all information about the MLE problem \cite[Section 2]{Hosten2004SolvingTL}, such as the {\em maximum likelihood} (ML) degree of the model.
While many results in the literature pertain to information derived from the likelihood ideal such as the ML degree \cites{hauenstein2013maximum, Amndola2017TheML, Michaek2021MaximumLD, amendola2021maximum, Gross2013MaximumLG},  the conditions on $\M$ that induce desirable properties for $\mathcal{L}_\M$ remain largely unknown. For instance, it is still unclear when $\mathcal{L}_\M$ is a complete intersection \cite[page 79]{LikelihoodGeometry}, a class of varieties that can be well understood using standard tools from algebraic geometry. This gap arises from a lack of tools to explicitly construct the vanishing ideal of $\mathcal{L}_\M$, referred to as the \textit{likelihood ideal} of $\M$. Although a general algorithm exists for finding the likelihood ideal \cite[Algorithm 6]{Hosten2004SolvingTL}, as we explain in \S \ref{sec:LikelihoodGeometry}, it is not explicit and proves computationally infeasible in practice. 

In this paper, we give a construction for the likelihood ideal for all toric models (Theorem \ref{thm:LCtoric}). Toric models are a widely used subclass of statistical model which includes all hierarchical log-linear models and undirected graphical models on discrete random variables. As we will show, this construction is substantially faster than previous methods. In the case of complete or joint independence models, Theorem \ref{thm:LCindependence} and Corollary \ref{cor:LCjointindependence} give an explicit Gr\"{o}bner basis. Using explicit implementations of these constructions, we work through several examples in \S \ref{sec:Applications}, and outline the main challenges to generalizing to other classes of models.

\subsection*{Data Availability} All code is available as the \textit{Macaulay2} package, \texttt{LikelihoodGeometry}.\footnote{ \href{https://github.com/johndcobb/LikelihoodGeometry}{https://github.com/johndcobb/LikelihoodGeometry}}

\subsection*{Outline} In \S\ref{sec:Background}, we give background on likelihood geometry and toric models and establish some key facts. In \S\ref{sec:Main}, we use these to prove the main results. Finally, in \S\ref{sec:Applications}, we use these results to compute the likelihood ideal of several interesting examples. 

\section{Background and Notation}\label{sec:Background}

\subsection{Likelihood Geometry}\label{sec:LikelihoodGeometry} A discrete statistical model $\M\subseteq \P^n_p$ gives rise to the likelihood correspondence variety $\mathcal{L}_\M\subseteq \P^n_p \times \P^n_u$, whose geometry reflects aspects of the MLE problem. We restrict our search for maximum likelihood estimates to the smooth locus $\M_{\text{reg}}$ and off the collection of $n+2$ hyperplanes $\H$ consisting of the $n+1$ coordinate hyperplanes $V(p_i)$ and the hyperplane $V(\sum p_i)$. This avoids the singularities coming from $\M$ and from the terms $\log(p_i)$ and $\log(\sum p_i)$ in the log--likelihood function $\log \ell_u$.

\begin{defn} \label{def:likelihoodcorrespondence}
The \textit{likelihood correspondence} $\mathcal{L}_\M$ is the Zariski closure of the set of  pairs $(p,u) \in \P_p^n \times \P_u^n$ such that $p\in \M$ is a critical point for the log--likelihood function. More precisely, 
    \[ \mathcal{L}_\M = \overline{\Set{(p,u) \in (\M_\text{reg} \setminus \H) \times \P_u^n \,:\, \dlog \res[\M]{\ell_u}(p) = 0 }}.\]
The vanishing ideal of $\mathcal{L}_\M$ is called the \textit{likelihood ideal} of $\M$.
\end{defn}

This variety is the universal family of all complex critical points of the likelihood function, from which the positive real points are maximum likelihood estimates on the probability simplex. In 2012, Huh proved that $\mathcal{L}_\M$ is an irreducible variety of dimension $n$ \cite[\S 2]{huh2013maximum}. The projection onto $\P^n_p$ realizes $\mathcal{L}_\M$ as a vector bundle over $\M_{\text{reg}} \setminus \H$ and the projection onto $\P^n_u$ is generically $k$-to-$1$, where $k$ is the maximum likelihood degree (\textit{ML degree}) of the model $\M$. In other words, the ML degree is the number of critical points of the likelihood function on $\M$ given a generic observed $u$ and is seen from the likelihood ideal through its multidegree \cite[page 100]{LikelihoodGeometry}. 
Example \ref{ex:LCofPn} computes the likelihood correspondence for the simplest case where $\M$ is the entire space $\P^n_p$.

\begin{example}\label{ex:LCofPn}
Let $\M=\P_p^n$ and fix $u\in \P^n_u$. We can take the exterior derivative of $\log \ell_u$ directly to give a one-form
\begin{equation*}
    \dlog(\ell_u(p)) = \sum_{i=0}^n \left(\frac{u_i}{p_i} - \frac{u_+}{p_+}\right) dp_i.
\end{equation*}
Here $p_+ = \sum_{i=0}^n p_i$ and $u_+ = \sum_{i=0}^n u_i$. This vector is zero exactly on the single point $p=u \in \P_p^n \setminus \H$. The likelihood correspondence $\mathcal{L}_\M$ has ML degree 1 and is given by the diagonal
$\Delta = \{ (u,u) \in \P^n_p \times \P^n_u \}$. Therefore the likelihood ideal inside the coordinate ring $\mathbb{C}[p_0,\dots, p_n,u_0,\dots, u_n]$ of $\P^n_p \times \P^n_u$ is given by
\begin{equation*}
     I_2 \begin{pmatrix}
        u_0 & \cdots & u_n\\
        p_0 & \cdots & p_n
    \end{pmatrix},
\end{equation*}
where $I_k(M)$ is the ideal of $k\times k$ minors of the matrix $M$.
\end{example}

Even without efficient methods for computing the likelihood ideal explicitly, a substantial literature exists to compute its properties \cites{hauenstein2013maximum, Gross2013MaximumLG, Amndola2017TheML, Michaek2021MaximumLD, amendola2021maximum, lindberg2023maximum}. Many of these results rely upon a general algorithm employing the method of Lagrange multipliers \cite[Algorithm 6]{Hosten2004SolvingTL}. We now describe this algorithm.

Let $I_\M = \langle f_0,\dots, f_r\rangle$ be the prime ideal of the model $\M\subseteq \P^n_p$, and $J(\M)$ be the Jacobian matrix of $I_\M$ of size $(r+1) \times (n+1)$. The following augmented Jacobian has $r+3$ rows and $n+1$ columns:
\[ \widetilde{J}(\M) = \begin{pmatrix}
    u_0 & u_1 & \cdots & u_n \\
    p_0 & p_1 & \cdots & p_n \\
    p_0 \frac{\partial f_0}{\partial p_0} & p_1 \frac{\partial f_0}{\partial p_1} & \cdots & p_n \frac{\partial f_0}{\partial p_n}\\
    p_0 \frac{\partial f_1}{\partial p_0} & p_1 \frac{\partial f_1}{\partial p_1} & \ddots  & p_n \frac{\partial f_1}{\partial p_n}\\
    \vdots & \vdots & \ddots & \vdots\\
    p_0 \frac{\partial f_r}{\partial p_0} & p_1 \frac{\partial f_r}{\partial p_1} & \dots & p_n \frac{\partial f_r}{\partial p_n}
\end{pmatrix} = \begin{pmatrix}
    \begin{pmatrix} u_0 & \cdots & u_n\end{pmatrix}\\
    \begin{pmatrix} p_0 & \cdots & p_n\end{pmatrix}\\
    J(\M)\cdot \text{diag}\begin{pmatrix} p_0 & \cdots & p_n\end{pmatrix}
\end{pmatrix}. \]

As further explained in \cite[Equation (2.3)]{hauenstein2013maximum}, the points in $\M$ with $\rank(\widetilde{J}(\M)) \leq \codim(\M)+1$ are solutions inside $\M$ of the system of equations
\[ \sum_{i=0}^n p_i = 1 \hspace{0.10 in} \text{and}  \hspace{0.10 in} \dlog(\ell_u(p)) = \lambda J(\M), \hspace{0.3in} \lambda \in \mathbb{C}. \] 
By the method of Lagrange multipliers, the ideal $I_\M + I_{\codim(\M)+1}(\widetilde{J}(\M))$ contains the likelihood ideal. However, in practice, the number of minors can be enormous and we must further saturate by the singular locus of $\M$ and the hyperplane arrangement $\H$ to remove extraneous solutions on the boundary. Example \ref{ex:HardyWeinberg} illustrates one case in which this can be computed directly. 

\begin{example}\label{ex:HardyWeinberg}
    Consider the problem of flipping a biased coin two times and recording the number of heads. Let $p=(p_0,p_1,p_2)$ represent the vector of probabilities of getting $0$ heads, $1$ head, and $2$ heads respectively. After repeating this experiment some number of times, we will obtain a corresponding data vector $u=(u_0,u_1, u_2)$. The possible probabilities for this experiment are defined by the \textit{Hardy--Weinberg curve} $\M = V(4p_0p_2-p_1^2) \subset \P_p^2$, a statistical model studied in population genetics with ML degree 1 \cite{Alg_stat_CB}. The likelihood correspondence is a surface $\mathcal{L}_\M \subset \P^2_p \times \P^2_u$ and can be computed using the method described above. The augmented Jacobian is given by
    \begin{equation*} 
    \widetilde{J}(\M) = \begin{pmatrix}
    u_0 & u_1 &  u_2 \\
    p_0 & p_1 & p_2 \\
    4p_0p_2  & -2p_1^2 & 4p_0p_2
\end{pmatrix}.
\end{equation*}
    Since $\M$ is smooth, we only need to saturate by the hyperplane arrangement $\H = V(p_0p_1p_2p_+)$. The likelihood ideal of $\M$ inside the coordinate ring $\C[p_0,p_1,p_2, u_0,u_1,u_2]$ of $\P_p^2 \times \P_u^2$ is exactly
    \[ (\langle 4p_0p_2-p_1^2 \rangle + I_2(\widetilde{J}(\M))): \langle p_0p_1p_2p_+ \rangle^\infty. \]
    In this case, $I_2(\widetilde{J}(\M))$ contributes $9$ generators. After saturation by $\H$, one sees that the likelihood ideal is generated by the following 3 polynomials:
    \begin{equation*}
        4p_{0}p_{2}-p_{1}^{2},\hspace{0.01in}\, 4\,p_{2}u_{0}-p_{1}u_{1}+2\,p_{2}u_{1}-2\,p_{1}u_{2}, \hspace{0.01in} \,2\,p_{1}u_{0}-2\,p_{0}u_{1}+p_{1}u_{1}-4\,p_{0}u_{2}.
    \end{equation*}
    Theorem \ref{thm:LCtoric} provides a faster method for computing the likelihood ideal for toric models such as the Hardy--Weinberg curve. This model will be revisited in Example \ref{ex:HardyWeinberg2}.
\end{example}

Even when the model is smooth and of low dimension, this method can fail to terminate in a reasonable time on a computer cluster (see Example \ref{ex:computationtime}). It is our goal to address this deficiency for the class of toric models.

\subsection{Contingency Tables}\label{sec:ContigencyTables}
In this paper, we consider statistical models analyzing relationships between discrete random variables as represented by {\em contingency tables}. Let $X_1,\dots, X_n$ denote discrete random variables, where $X_i$ takes values (or outcomes) in the set $[d_i] = \{1,2,\dots, d_i\}$. We let $V_1 \times V_2 \times \cdots \times V_n$ represent the vector space of $n$-dimensional contingency tables of format $d_1 \times d_2 \times \cdots \times d_n$. We introduce the indeterminates $p_{i_1\dots i_n}$ to represent the joint probability of event $X_1=i_1$, $X_2=i_2$, \dots, $X_n=i_n$. These indeterminates generate the coordinate ring $\mathbb{C}[p_{i_1\cdots i_n}]$ of homogeneous polynomials for the projective space $\P(V_1\otimes V_2\otimes \cdots \otimes V_n)= \P_p^{D-1}$ of all contingency tables representing possible joint probability distributions, where $D= d_1\cdots d_n$. Similarly, the indeterminates $u_{i_1,\dots, i_n}$ generate the coordinate ring of a second projective space $\P_u^{D-1}$ of contingency tables representing possible joint frequency distributions. In keeping with standard notation, we extend the notation in Example \ref{ex:LCofPn} by using $+$ in the $k$-th place in the subscript of $p_{i_1\dots i_n}$ or $u_{i_1\dots i_n}$ to denote the \textit{marginal sums} over all states of the random variable $X_k$, i.e $p_{+i_2\dots i_n} = \sum_{j=1}^{d_1} p_{j i_2 \dots i_n}$. 

A \textit{conditional independence statement} about $X_1,\dots, X_n$ has the form $A \indep B \,|\, C$ where $A$, $B$, and $C$ are pairwise disjoint subsets of $X= \{X_1,\dots, X_n\}$. In words, $A \indep B \,|\, C$ professes that $A$ is independent of $B$ given $C$. Any such statement gives rise to a model whose vanishing ideal $I_{A \indep B \,|\, C}$ is prime and generated by quadrics \cite[Theorem 8.1, Lemma 8.2]{sturmfels2002solving}. In general, a collection of such statements $A^{(1)} \indep B^{(1)} \, | \, C^{(1)}, \dots, A^{(m)} \indep B^{(m)} \, | \, C^{(m)}$ gives rise to a model 
\[\M = V(I_{A^{(1)} \indep B^{(1)} \, | \, C^{(1)}}) \cap \cdots \cap V(I_{A^{(m)} \indep B^{(m)} \, | \, C^{(m)}}). \] 
The properties that the conditional independence statements impart onto $\M$ are well-studied \cite[Chapter 8]{sturmfels2002solving}. For example, $I_\M$ is a binomial ideal if and only if the statements $A^{(i)} \indep B^{(i)} \, |\,  C^{(i)}$ are \textit{saturated}, i.e. $A^{(i)} \cup B^{(i)} \cup  C^{(i)} = X$ for all $i$. 

\begin{example}\label{ex:2x2Ind}
Let $X_1$ and $X_2$ be binary random variables, represented by the space of $2\times 2$ contingency tables of joint probability and frequency distributions:
\begin{equation*}
    P = \begin{pmatrix}
    p_{11} & p_{12}\\
    p_{21} & p_{22}
\end{pmatrix} \subseteq \P^3_p, \hspace{0.2in} 
 U = \begin{pmatrix}
    u_{11} & u_{12}\\
    u_{21} & u_{22}
\end{pmatrix} \subseteq \P^3_u.
\end{equation*}
Consider the complete independence model given by $\M=V(I_{X_1\indep X_2})$. One can compute that the vanishing ideal of $\M$ is $I_\M = \langle p_{11}p_{22}-p_{12}p_{21}\rangle$. Theorem \ref{thm:LCindependence} says that a Gröbner basis of the likelihood ideal in the coordinate ring $\C[p_{11},p_{12},p_{21},p_{22},u_{11},u_{12},u_{21},u_{22}]$ of $\P^3_p \times \P^3_u$ is
\begin{equation*}
    I_2\begin{pmatrix}
        p_{11} & p_{12} & u_{1+}\\
        p_{21} & p_{22} & u_{2+}
    \end{pmatrix} + 
    I_2\begin{pmatrix}
        p_{11} & p_{21} & u_{+1}\\
        p_{12} & p_{22} & u_{+2}
    \end{pmatrix}.
    \qedhere
\end{equation*}

\end{example}

Now that we have set up our data type, for the remainder of this section we turn to describing the statistical models in play.

\subsection{Toric Models and Sufficient Statistics}\label{sec:toric}

A model $\M$ is \textit{toric} if $I_\M$ is binomial and prime, that is, $\M$ is a toric model. Toric models are among the most thoroughly studied statistical models \cites{on_tor, Alg_stat_CB, sullivant}. All statistical models we work with are of this type; the following inclusions to illustrate how these statistical models relate to each other:

\begin{center}
    Toric
    $\supset$ (Hierarchical) Log--linear $\supset$ Undirected Graphical $\supset$ Independence.
\end{center}

A toric variety $\M\subseteq \P_p^n$ has the following classical combinatorial representation. Fix an $m \times (n+1)$ integer matrix $A=(a_{ij})$ such that all column sums are equal, and let $p$ and $u$ be the vector of indeterminates of the coordinate rings of $\P^n_p$ and $\P^n_u$, respectively. The corresponding toric ideal $I_A$ is generated by all binomials $p^v-p^w$ with $v,w \in \N^{n+1}$ such that $Av = Aw$, where $p^v = p_0^{v_0}p_1^{v_1}\cdots p_n^{v_n}$. Writing $A_i$ for $i$-th column of $A$, the ideal $I_A$ cuts out the variety given by the closure of the image of the map 
\begin{center}
\begin{tikzcd}[row sep= tiny]
    (\C^*)^{m} \ar[r] & \P^n\\
\theta \ar[r, mapsto] & {[\theta^{A_0}: \theta^{A_1}: \dots: \theta^{A_n}]}.
\end{tikzcd}
\end{center}
Substituting this parameterization into the likelihood function on the probability simplex, one sees that the toric model depends only on the vector $Au$:
\[ \ell_u(\theta) = \prod\limits_{i=0}^n(\theta^{A_i})^{u_i} = \theta^{Au}. \]
This is the vector of \textit{minimal sufficient statistics} for $\M$ and is all that is required to solve the MLE problem \cite[Definition 1.1.10]{drton2009}.

\begin{example}
    Consider the $2$-way binary complete independence model given in Example \ref{ex:2x2Ind}. This is an example of a toric model with the corresponding $A$ matrix given by
    \begin{equation*}
    A=\begin{pmatrix}
        1 & 1 & 0 & 0\\
        0 & 0 & 1 & 1\\
        1 & 0 & 1 & 0\\
        0 & 1 & 0 & 1
    \end{pmatrix}.
    \end{equation*}
    The sufficient statistics are exactly the marginal sums: $Au = \begin{pmatrix}
            u_{1+} & u_{2+} & u_{+1} & u_{+2}
        \end{pmatrix}.$
    As  seen in Example \ref{ex:2x2Ind}, the entries of $Au$ are exactly what is needed to compute the likelihood ideal.
\end{example}

\subsection{Log--linear and Undirected Graphical Models}\label{sec:graphical} 
Hierarchical log--linear models are a class of toric models. Let $\mc{X} = [d_1] \times \cdots \times [d_n]$ denote the joint state space of the discrete random variables $X_1,\dots, X_n$. A \textit{hierarchical log--linear model} (or simply \textit{log--linear model}) is defined by a collection $\{G_1,\dots, G_g\}$ of non-empty subsets of $X = \{X_1,\dots, X_n\}$ called \textit{generators}.  
Given such a collection, we can construct the corresponding matrix $A$ by indexing the columns by elements of $\mc{X}$ and indexing the rows by all possible states of the variables in each generator $G_i$. The entries of the matrix $A$ are either $0$ or $1$. An entry of matrix $A$ at column $s \in \mc{X}$ (representing $X_1 = s_1, X_2 = s_2, \dots, X_n = s_n$) and row $X_{i_1} = t_1, X_{i_2}=t_2, \dots, X_{i_m} = t_m$ is 1 if and only if $s_{i_k} = t_k$ for all $1\leq k \leq m$. This is best explored by an example.

\begin{example}\label{ex:3chain}
    Let $X = \{X_1,X_2,X_3\}$ be binary random variables and the generators be $\{X_1,X_2\}$ and $\{X_2,X_3\}$. The corresponding matrix $A$ is:
    \begin{equation*}
        \begin{blockarray}{ccccccccc}
        & 111 & 112 & 121 & 122 & 211 & 212 & 221 & 222\\
        \begin{block}{c(cccccccc)}
            (X_1=1, X_2=1) & 1 & 1 & 0 & 0 & 0 & 0 & 0 & 0\\
            (X_1=1, X_2=2) & 0 & 0 & 1 & 1 & 0 & 0 & 0 & 0\\
            (X_1=2, X_2=1) & 0 & 0 & 0 & 0 & 1 & 1 & 0 & 0\\
            (X_1=2, X_2=2) & 0 & 0 & 0 & 0 & 0 & 0 & 1 & 1\\
            (X_2=1, X_3=1) & 1 & 0 & 0 & 0 & 1 & 0 & 0 & 0\\
            (X_2=1, X_3=2) & 0 & 1 & 0 & 0 & 0 & 1 & 0 & 0\\
            (X_2=2, X_3=1) & 0 & 0 & 1 & 0 & 0 & 0 & 1 & 0\\
            (X_2=2, X_3=2) & 0 & 0 & 0 & 1 & 0 &0  & 0 & 1\\
        \end{block}
        \end{blockarray}\ \ .
        \end{equation*}
    One can compute this model has minimal sufficient statistics $u_{ij+}$ and $u_{+jk}$ and ML degree 1. It is shown in \cite[Table 8.13]{agresti:categorical} that the maximum likelihood estimate is given by
\[\widehat{p}_{ijk}=\frac{u_{ij+}u_{+jk}}{u_{+++}u_{+j+}}. \qedhere \]
\end{example}

A special case of log--linear models are \textit{undirected graphical models}. The undirected graphical model for a graph $\mc{G}$ is the log--linear model on $X$ in which the generators are cliques (maximal complete subgraphs) of $\mc{G}$. Example \ref{ex:3chain} corresponds to the  graph
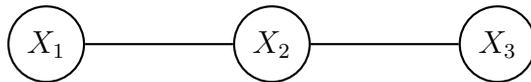
\begin{figure}[H]
    \centering
\begin{tikzpicture}[node distance={30mm}, thick, main/.style = {draw, circle}]
\node[main] (3) {$X_1$}; 
\node[main] (4) [right of=3] {$X_2$};
\node[main] (5) [right of=4] {$X_3$}; 

\draw [-](3) -- (4);
\draw [-](4) -- (5);

\end{tikzpicture}
\end{figure}
\noindent and the conditional independence statement $X_1 \indep X_3 \, | \, X_2$. It is known that the undirected graphical model of a graph $\mc{G}$ has ML degree 1 if and only if $\mc{G}$ is decomposable -- that is, every cycle of length 4 or more has a chord \cite[Theorem 4.4]{on_tor}.

\subsection{Complete and Joint Independence Models}\label{sec:independence}

The undirected graphical model on $X$ where the graph has no edges is the \textit{complete independence model}, which models pairwise independence between all random variables. Since this graph is decomposable, the ML degree of the complete independence model is 1 and by \cite[Theorem 4.18]{lauritzen1996graphical},  the maximum likelihood estimate has the closed form 
\begin{equation*}
    \widehat{p}_{i_1 \dots i_n} = (\prod_{k=1}^n u_{+\cdots +i_k+\cdots +})/(u_{+\cdots +})^n
\end{equation*} 
and thus the likelihood correspondence has an explicit parametrization given by
\begin{equation}\label{eqn:IndependenceOptimalMLE}
\mathcal{L}_\M = \left\{ \left(\widehat{p}_{i_1,\dots, i_n} = \prod_{k=1}^n u_{+\cdots +i_k+\cdots +}/(u_{+\cdots +})^n, u_{i_1,\dots,i_n}\right) \in \P^{D-1}_p \times \P^{D-1}_u \right\}.
\end{equation} 

The complete independence model $\M \subset \P(V_1\otimes V_2 \otimes \cdots \otimes V_n)$ is the subset of tensors in $X$ that have rank one, that is, all $2$-minors of the generic $d_1\times \cdots \times d_n$ hypermatrix $P = \left( p_{i_1\dots i_n} \right)_{1\leq i_k \leq d_k}$ \cite{grone1977decomposable}. This is exactly the \textit{Segre variety} $\P(V_1) \times \cdots \times \P(V_n)$, whose ideal is the $2\times 2$ minors of the matrices $P_i = V_i \otimes \left( \bigotimes_{j\neq i} V_j\right)$ of size $d_i \times D/d_i$ defined for each random variable $X_i$ \cite[Corollary 1.6.1]{ha2002box}.

\begin{example}\label{ex:234table}
    Consider the complete independence model $\M$ on $2 \times 3 \times 4$ contingency tables $P = \left( p_{ijk} \right)$. The ideal $I_\M$ is generated by the $2\times 2$ minors of the following matrices:
    \[P_1=\begin{pmatrix}
        p_{111} & p_{112} & p_{113} & p_{114} & p_{121} & p_{122} & p_{123} & p_{124} & p_{131} & p_{132} & p_{133} & p_{134} \\
        p_{211} & p_{212} & p_{213} & p_{214} & p_{221} & p_{222} & p_{223} & p_{224} & p_{231} & p_{232} & p_{233} & p_{234} \\
    \end{pmatrix}\]
    \[P_2=\begin{pmatrix}
        p_{111} &p_{112} & p_{113} & p_{114} & p_{211} & p_{212} & p_{213} & p_{214} \\
        p_{121} &p_{122} & p_{123} & p_{124} & p_{221} & p_{222} & p_{223} & p_{224} \\
        p_{131} &p_{132} & p_{133} & p_{134} & p_{231} & p_{232} & p_{233} & p_{234} \\
    \end{pmatrix}\]
    \[P_3=\begin{pmatrix}
        p_{111} &p_{121} & p_{131} & p_{211} & p_{221} & p_{231} \\
        p_{112} &p_{122} & p_{132} & p_{212} & p_{222} & p_{232} \\
        p_{113} &p_{123} & p_{133} & p_{213} & p_{223} & p_{233} \\
        p_{114} &p_{124} & p_{134} & p_{214} & p_{224} & p_{234} \\
    \end{pmatrix}. \qedhere\]
\end{example}

The following lemma collects these facts to give a description of likelihood ideal of complete independence models that will be crucial in defining the likelihood ideal.

\begin{lemma}\label{lemma:independencesegreideal}
    The vanishing ideal $I_\M$ of the complete independence model $\M \subseteq \P_p^{D -1}$ arising from a $d_1\times \cdots \times d_n$ contingency table is $\sum_{i=1}^n I_2(P_i)$. Furthermore, $\M$ is isomorphic to the Segre variety $\P^{d_1-1} \times \cdots \times \P^{d_n-1}$ and therefore is of dimension $\sum_{i=1}^n(d_i-1)$ and degree $(\sum_{i=1}^n (d_i-1))!/\prod_{i=1}^n (d_i-1)!$.
\end{lemma}

More generally, a \textit{joint independence model} on $X$ is the undirected graphical model where every connected component of the associated graph is a complete graph. 
Joint independence models can always be reduced to a complete independence model with fewer variables \cite[page 318]{agresti:categorical}. This identification with complete independence models leads to the following extension of Lemma \ref{lemma:independencesegreideal}.

\begin{lemma}\label{lemma:jointiscomplete}
    Let $\M\subseteq \P^{D-1}_p$ be a joint independence model whose complete graphs are given by a partition $P = \{P_1,\dots,P_q\}$ of the variables $\{X_1,\dots, X_n\}$. Then $\M$ is isomorphic to the product of $q$ projective spaces $\P(\bigotimes_{X_k \in P_i} V_k)$ for each partition $P_i$ under the Segre embedding.
\end{lemma}

\begin{example}\label{ex:join_indep}
    Consider discrete random variables $X_1$, $X_2$, and $X_3$ corresponding to the space of $2\times 3 \times 4$ contingency tables. The joint independence model $X_1\indep \{X_2,X_3\}$ is represented by the graph
 
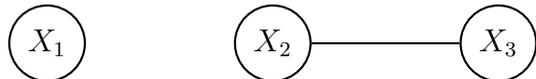
\begin{figure}[H]
\centering
\begin{tikzpicture}[node distance={30mm}, thick, main/.style = {draw, circle}]
\node[main] (3) {$X_1$}; 
\node[main] (4) [right of=3] {$X_2$};
\node[main] (5) [right of=4] {$X_3$}; 

\draw [-](4) -- (5);

\end{tikzpicture}.
\end{figure}
\noindent The corresponding $A$ matrix has $14$ rows and $24$ columns. The minimal sufficient statistics are $u_{i++}$ corresponding to the complete graph on $X_1$, and $u_{+jk}$ corresponding to the complete graph on $X_2$ and $X_3$. By considering $X_2$ and $X_3$ as a single random variable $W$ with 12 possible outcomes, this model is isomorphic to the complete independence model $X_1 \indep W$ on the space of $2\times 12$ contingency tables. That is, the corresponding $A$ matrix is the same (albeit with different row and column indexes) and the minimal sufficient statistics are given by $u_{i+}$ and $u_{+i}$.
\end{example}

\section{Computing the Likelihood Correspondence} \label{sec:Main}

With the models of interest defined, we move on to our main results. Theorem \ref{thm:LCtoric} constructs the likelihood ideal of any toric model. Theorem \ref{thm:LCindependence} gives a Grobner basis for the likelihood ideal of complete independence models, and Corollary \ref{cor:LCjointindependence} extends this result to joint independence models using Lemma \ref{lemma:jointiscomplete}. The construction for any toric model is fairly straightforward due to some previous work; The main idea is to use Birch's theorem to give a direct description of the likelihood correspondence on an affine subset, as suggested in \cite[Proposition 3.10]{LikelihoodGeometry}. 

\begin{theorem}\label{thm:LCtoric}
   Let $\M\subseteq \P^n_p$ be a toric model defined by the matrix $A$. Define 
    \[ M = \begin{pmatrix}
        u_0 & p_0\\
        \vdots & \vdots\\
        u_n & p_n
    \end{pmatrix}. \] Then $(I_A + I_2(AM) ): \langle p_0p_1\cdots p_np_+\rangle^\infty $ is the likelihood ideal of $\M$.
\end{theorem}
\begin{proof}
    Let  $I= (I_A + I_2(AM) ): \langle p_0p_1\cdots p_np_+\rangle^\infty$. First, we show that $\mathcal{L}_\M$ agrees with $V(I)$ after restricting to $(\P_p^n \setminus \H) \times \P^n_u$, or equivalently, that $V(I_2(AM))$ agrees with $\mathcal{L}_\M$ on $(\M \setminus \H) \times \P^n_u$. Note that a point $(p,u)$ is in $V(I_2(AM))$ whenever $AM = \begin{pmatrix} Au & Ap \end{pmatrix}$ drops rank, or at all points $(p,u)$ such that $Ap$ is proportional to the sufficient statistic $Au$. By Birch's theorem \cite[Theorem 1.10]{Alg_stat_CB}, these are exactly the maximum likelihood estimates for $u$ and therefore define $\mathcal{L}_\M$. Since $\mathcal{L}_\M$ agrees with $V(I)$ on $(\P_p^n \setminus \H) \times \P^n_u$, any non-reduced structure of $V(I)$ is supported on $\mathcal{H}$ since $\mathcal{L}_\M$ is prime. Since these hyperplanes are specifically saturated out from $I$, it must be radical. Since both the likelihood ideal and $I$ are radical and equal on a dense open set, $V(I)$ and $\mathcal{L}_\M$ are also equal after taking closure.
\end{proof}

As will be shown experimentally in \S \ref{sec:Applications}, this can be much faster than the Lagrange multipliers method due to only requiring  $2 \times 2$ minors of a single $m \times 2$ matrix. For complete and joint independence models, we improve upon Theorem \ref{thm:LCtoric} by directly giving a Gröbner basis.

Consider the complete independence model $\M \subset \P_p^{D-1}$ inside the projective space on $V_1\times \cdots \times V_n$, the vector space of all $d_1 \times \cdots \times d_n$ contingency tables coming from discrete random variables $X_1,\dots, X_n$. We will show that the likelihood ideal of $\M$ is generated by $2\times 2$ minors of matrices obtained by augmenting $P_i$, the matrices in Lemma \ref{lemma:independencesegreideal}. Let $M_i$ be the matrix obtained from $P_i$ by appending the column of marginal sums over all states of $X_i$ (see Example \ref{ex:234Ms}). Theorem \ref{thm:LCindependence} and Corollary \ref{cor:LCjointindependence} show that the Gröbner basis of the likelihood ideal of complete and joint independence models come from $2\times 2$ minors of $M_i$.

\begin{theorem}\label{thm:LCindependence}
    The ideal $I=\sum_{i=1}^n I_2(M_i)$ is the prime likelihood ideal of the complete independence model $\M$. Furthermore, the generators of $I$ are a Gr\"{o}bner basis under the lexicographic term order for $p$'s followed by the lexicographic term order for the $u$'s.
\end{theorem}

\begin{example}\label{ex:234Ms}
    Consider the complete independence model $\M$ from Example \ref{ex:234table}. By augmenting each $P_i$ accordingly we get:
    \[M_1=\begin{pmatrix}
        p_{111} & p_{112} & p_{113} & p_{114} & p_{121} & p_{122} & p_{123} & p_{124} & p_{131} & p_{132} & p_{133} & p_{134} & u_{1++}\\
        p_{211} & p_{212} & p_{213} & p_{214} & p_{221} & p_{222} & p_{223} & p_{224} & p_{231} & p_{232} & p_{233} & p_{234} & u_{2++}
    \end{pmatrix}\]
    \[M_2=\begin{pmatrix}
        p_{111} &p_{112} & p_{113} & p_{114} & p_{211} & p_{212} & p_{213} & p_{214} & u_{+1+} \\
        p_{121} &p_{122} & p_{123} & p_{124} & p_{221} & p_{222} & p_{223} & p_{224} & u_{+2+}\\
        p_{131} &p_{132} & p_{133} & p_{134} & p_{231} & p_{232} & p_{233} & p_{234} & u_{+3+}
    \end{pmatrix}\]
    \[M_3=\begin{pmatrix}
        p_{111} &p_{121} & p_{131} & p_{211} & p_{221} & p_{231} & u_{++1} \\
        p_{112} &p_{122} & p_{132} & p_{212} & p_{222} & p_{232} & u_{++2} \\
        p_{113} &p_{123} & p_{133} & p_{213} & p_{223} & p_{233} & u_{++3} \\
        p_{114} &p_{124} & p_{134} & p_{214} & p_{224} & p_{234} & u_{++4} 
    \end{pmatrix}.\]
    The likelihood ideal of $\M$ is generated by the $2\times 2$ minors of these matrices. This is a Gr\"obner basis under the term order $p_{111} > \cdots > p_{234} > u_{111} > \cdots > u_{234}$.
\end{example}

\begin{proof}[Proof of Theorem \ref{thm:LCindependence}]
    We start by giving a Grobner basis of $I$. The given term order $<$ induces a diagonal term order upon each matrix $M_i$, that is, it respects the rows and columns of each matrix $M_i$. By \cite[Theorem 1]{sturmfels1990grobner}, the ideal of $2\times 2$ minors of $M_i$ forms a (reduced) Gr\"{o}bner basis for each ideal $I_2(M_i)$. 
    By \cite[Lemma 1.3(c)]{conca1996gorenstein}, the Gr\"{o}bner basis of the sum $\sum_{k=1}^n I_2(M_k)$ is obtained by taking the union of the Gr\"{o}bner bases of each term. Therefore, the set of $2\times 2$ minors of $M_i$ for all $i$ forms a Gr\"{o}bner basis with respect to $<$. 
    
    To show that $I=I(\mathcal{L}_\mathcal{M})$, 
it suffices to show $I\subseteq I(\mathcal{L}_\mathcal{M})$, $\dim(I) = \dim(I(\mathcal{L}_\mathcal{M}))$, and that $I$ is prime. Since plugging the parametrization (\ref{eqn:IndependenceOptimalMLE}) into $M_i$ causes all the $2\times 2$ minors to vanish, we have $I\subseteq I(\mathcal{L}_\mathcal{M})$. In Lemma \ref{lemma:dimI}, we compute $\dim(I)$ directly using the Gr\"{o}bner basis above and a known formula for computing the dimension of monomial ideals.\footnote{Alternatively, one can prove $V(I)=\mathcal{L}_\M$ on $(\P_p^{D-1} \setminus \mathcal{H}) \times  \P_u^{D-1}$ by using the parametric description of $\M$.} It only remains to show that $I$ is prime. 
 
    Since the Grobner basis is square-free, $I$ is radical. To show that $I$ is prime, we show that $V(I)$ is a vector bundle over $\M$ which is prime. It suffices to show that the fiber over any point in $\M$ has codimension $\dim \M = \sum_{i=1}^n(d_i-1)$ inside $V(I)$. Since $\M \cong \P(V_1) \times \cdots \times \P(V_n)$ comes with a transitive group action by the product of projective general linear groups $PGL(V_1) \times \cdots \times PGL(V_n)$, we may check the fiber dimension over the single point $p_{1,\dots, 1} = 1$ and all other $p_{i_1,\dots, i_n} = 0$. The codimension of this fiber is then equal to the dimension of the ideal of $2\times 2$ minors of all $M_i$ after plugging in this single point. From each matrix $M_i$, we get exactly $d_i-1$ nonzero $2\times 2$ minors for a total of $\sum_{i=1}^n (d_i - 1)$ generators. We can check these generate an ideal whose height is $\sum_{i=1}^n (d_i - 1)$, for example by passing to the initial ideal 
    \begin{equation*}
        \{ u_{i_1,\dots, i_n} \, | \, i_k = 1 \text{ for all but one index } k \},
    \end{equation*} 
    which is just the intersection of $\sum_{i=1}^n (d_i - 1)$ coordinate hyperplanes.
\end{proof}

\begin{lemma}\label{lemma:dimI}
 Let $I$ be as in Theorem \ref{thm:LCindependence}. The initial ideal $\text{in}_<(I)$ has dimension $D-1$ as a projective variety.
\end{lemma}
\begin{proof}
    Split off the $2 \times 2$ minors of $I$ coming from the submatrix $P_i$ from the $2\times 2$ minors that involve terms with $u$: 
    \begin{align*}
        I &= \sum_{i=1}^n I_2(P_i) + \sum_{i=1}^n \sum_{k=1}^{D/d_i} I_2(M_i\cdot(e_k+e_{1+D/d_i})))\\
        &= I_\M + \sum_{i=1}^n \sum_{k=1}^{D/d_i} I_2(M_i\cdot(e_k+e_{1+D/d_i}))).
    \end{align*}
    Here, $e_k$ refers to the standard basis vector with $1$ in the $k$-th entry and $0$ elsewhere, and the second term simply collects the $2\times 2$ minors coming from  pairing each column in $M_i$ with the last column of minimal sufficient statistics. The codimension of the vanishing of a monomial ideal such as $V(\text{in}_<(I))$ as a projective variety is the minimum cardinality of a subset $S$ of the variables of $\C[p_{i_1\dots i_n},u_{k_1\dots k_n}]$ such that at least one of the variables in $S$ shows up in each monomial generator of $\text{in}_<(I)$  \cite[Chapter 9]{cox1994ideals}. Since $I_\M \subset I$, $|S|$ must be at minimum $\codim(\M) = D - 1 - \sum_{i=1}^n (d_i-1)$ to cut out $\text{in}_<(I_\M)$. That is, $S$ must contain a subset $J$ of the variables in the hypermatrix $P=\left( p_{i_1 \dots i_n} \right)_{1\leq i_k \leq d_k}$ of size $|J| = D - 1 - \sum_{i=1}^n (d_i-1)$ that covers the $2\times 2$ minors of variables coming from $\text{in}_<(I_\M)$. Since the ideal of $2\times 2$ minors does not depend on permuting the dimensions of $P$, without loss of generality we can let $J$ be the subset 
    \begin{equation*}
        \{ p_{i_1 \dots i_n} \, | \, i_j < d_j, i_k < d_k \text{ for at least two indices } j \text{ and } k \}.
    \end{equation*} 
    With $J\subset S$, the only monomials in $\text{in}_<(I)$ containing variables not present in $S$ are exactly coming from the $2\times 2$ minors of the last two columns of each $M_i$:
    \begin{equation*}
        Q_i = \text{in}_<\left(I_2(M_i\cdot (e_{D/d_i} + e_{1+D/d_i}))\right).
    \end{equation*}
    Furthermore, the variables involved in the generators of $Q_i$ are disjoint to those of $Q_j$ for $i \neq j$. The ideal $Q_i$ is exactly the edge ideal coming from the following bipartite graph:

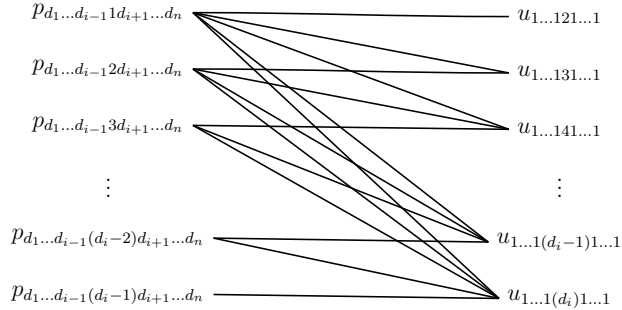
\begin{figure}[H]
    \centering
\scalebox{0.75}{\begin{tikzpicture}[node distance={30mm}, thick, main/.style = {circle}]
\node (3) at (0,0) {$p_{d_1\dots d_{i-1} 1  d_{i+1} \dots d_n}$}; 
\node (4) at (0,-1) {$p_{d_1\dots d_{i-1} 2  d_{i+1} \dots d_n}$}; 
\node (5) at (0,-2) {$p_{d_1\dots d_{i-1} 3  d_{i+1} \dots d_n}$};
\node (6) at (0,-3) {$\vdots$};
\node (7) at (0,-4) {$p_{d_1\dots d_{i-1} (d_i-2)  d_{i+1} \dots d_n}$};
\node (8) at (0,-5) {$p_{d_1\dots d_{i-1} (d_i-1)  d_{i+1} \dots d_n}$};
\node (9) at (8,-0.07) {$u_{1\dots 1 2  1 \dots 1}$};
\node (10) at (8,-1.07) {$u_{1\dots 1 3  1 \dots 1}$};
\node (11) at (8,-2.07) {$u_{1\dots 1 4  1 \dots 1}$};
\node (12) at (8,-3) {$\vdots$};
\node (13) at (8,-4.07) {$u_{1\dots 1 (d_i-1)  1 \dots 1}$};
\node (14) at (8,-5.07) {$u_{1\dots 1 (d_i)  1 \dots 1}$};

\draw [-](3) to [out=0,in=180] (9);
\draw [-](3) to [out=0,in=180,looseness=0] (10);
\draw [-](3) to [out=0,in=180,looseness=0] (11);
\draw [-](3) to [out=0,in=180,looseness=0] (13);
\draw [-](3) to [out=0,in=180,looseness=0] (14);
\draw [-](4) to [out=0,in=180] (10);
\draw [-](4) to [out=0,in=180,looseness=0] (11);
\draw [-](4) to [out=0,in=180,looseness=0] (13);
\draw [-](4) to [out=0,in=180,looseness=0] (14);
\draw [-](5) to [out=0,in=180,looseness=0] (11);
\draw [-](5) to [out=0,in=180,looseness=0] (13);
\draw [-](5) to [out=0,in=180,looseness=0] (14);
\draw [-](7) to [out=0,in=180,looseness=0] (13);
\draw [-](7) to [out=0,in=180,looseness=0] (14);
\draw [-](8) to [out=0,in=180,looseness=0] (14);
\end{tikzpicture}}
\end{figure}
    
    The variables we must add to $S$ are exactly a minimal set of nodes of the above graph such that deleting them removes all the edges. For the graph above, we must remove $d_i-1$ nodes; for example, all the nodes on the left. Doing this for each $Q_i$ adds $\sum_{i=1}^n(d_i-1)$ new variables to $S$ so that $|S| = D-1$. This means that the codimension of $V(\text{in}_<(I))$ as a projective variety is $D-1$, as desired.
\end{proof}

Since joint independence models can be represented as a complete independence model with fewer variables, using Lemma \ref{lemma:jointiscomplete} we can extend Theorem \ref{thm:LCindependence} to joint independence models. See Example \ref{ex:join_dep2} for an example doing this explicitly.

\begin{corollary}\label{cor:LCjointindependence}
    Theorem \ref{thm:LCindependence} holds for joint independence models.
\end{corollary}
\begin{proof}
    Let $\M$ be a joint independence model. By Lemma \ref{lemma:jointiscomplete}, $\M$ is isomorphic to a product of projective spaces under the Segre embedding and therefore is isomorphic to a complete independence model on fewer variables by Lemma \ref{lemma:independencesegreideal}. 
\end{proof}

\section{Applications and Examples}\label{sec:Applications}

In this section, we demonstrate algorithmic advantages afforded by our main results and give several examples of how these results may be used in practice. At the end, we include an example that demonstrates the challenges of generalizing our method to more general classes of models. The following example organizes the computation times of the various algorithms introduced in this paper for the complete independence model on $2$-way and $3$-way contingency tables.

\begin{example} \label{ex:computationtime}
The following tables organize the computation times for $2$-way and $3$-way complete independence models using Theorem \ref{thm:LCindependence} (labeled I) and Theorem \ref{thm:LCtoric} (labeled T). Our implementation of method T is made faster by saturating only by $p_+$ and the output was verified for correctness separately. The code used was written in \textit{Macaulay2}~\cite{M2} and all computations were done using a cluster node equipped with 2 Intel Xeon Gold 6130 CPUs with 192 GB of memory.

\begin{table}[H]\scalebox{0.9}{
\begin{tabular}{|c|c|c|c|c|c|c|}
\hline
\multicolumn{7}{|c|}{$i \times j$ computation times (s)}\\\hline\hline
&2&3&4&5&6&\text{Method}\\
 \hline
 \multirow{2}{*}{{2}}&{.0042050}&{.0026273}&{.0030742}&{.0037202}&{.0042469}&I\\\cline{7-7}
     &{.0126488}&{.0121614}&{.019614}&{.0155284}&{.0170114}&T\\\hline
     \multirow{2}{*}{{3}}&{}&{.0030728}&{.0038790}&{.0052260}&{.0058382}&I\\\cline{7-7}
     &{}&{.0142195}&{.0172784}&{.0271773}&{.0276863}&T\\\hline
     \multirow{2}{*}{{4}}&{}&{}&{.0049143}&{.0064623}&{.0142544}&I\\\cline{7-7}
     &{}&{}&{.0198774}&{.0350445}&{.0484354}&T\\\hline
     \multirow{2}{*}{{5}}&{}&{}&{}&{.0134698}&{.0167110}&I\\\cline{7-7}
     &{}&{}&{}&{.0443950}&{.0581353}&T\\\hline
     \multirow{2}{*}{{6}}&{}&{}&{}&{}&{.0233014}&I\\\cline{7-7}
     &{}&{}&{}&{}&{.0864135}&T\\\hline
     \end{tabular}}
     \caption*{}
     \label{tab:sdf}
 \end{table}
\vspace*{-1cm}
\begin{table}[H]
\scalebox{0.9}{
\begin{tabular}{|c|c|c|c|c|c|c|c|}
\hline
\multicolumn{8}{|c|}{$2\times j \times k$ computation times (s)}\\\hline\hline
\multicolumn{2}{|c|}{}&2&3&4&5&6&\text{Method}\\
 \hline
 \multirow{10}{*}{{2}}&\multirow{2}{*}{{2}}&{.0055735}&{.0053818}&{.0068220}&{.0157749}&{.0162699}&I\\\cline{8-8}
     &&{.0313907}&{.0669933}&{.0582959}&{.0984729}&{.1559380}&T\\\cline{2-8}
     &\multirow{2}{*}{{3}}&{}&{.0119334}&{.0099675}&{.0280001}&{.0310153}&I\\\cline{8-8}
     &&{}&{.0659359}&{.1769630}&{.4047150}&{1.56129}&T\\\cline{2-8}
     &\multirow{2}{*}{{4}}&{}&{}&{.0241910}&{.0362822}&{.0538377}&I\\\cline{8-8}
     &&{}&{}&{.571813}&{2.20428}&{8.87356}&T\\\cline{2-8}
     &\multirow{2}{*}{{5}}&{}&{}&{}&{.0485734}&{.0844884}&I\\\cline{8-8}
     &&{}&{}&{}&{13.6736}&{14.4752}&T\\\cline{2-8}
     &\multirow{2}{*}{{6}}&{}&{}&{}&{}&{.1191160}&I\\\cline{8-8}
     &&{}&{}&{}&{}&{44.5459}&T\\\hline
     \end{tabular}}
     \caption*{}
     \label{tab:sdf}
 \end{table}
 \vspace*{-1cm}
 
 \noindent In comparison, the Lagrange method from \S\ref{sec:LikelihoodGeometry} took 120.638 seconds to compute just the $2\times 2$ complete independence model. We note that for both the $2 \times 3$ and the $3 \times 3$ models the Lagrange method did not terminate after several days. The method provided by Theorem \ref{thm:LCindependence} outperforms the method offered by Theorem \ref{thm:LCtoric} drastically for complete independence models on $2\times j \times k$ contingency tables.
\end{example}

\begin{example}\label{ex:join_dep2}
   Consider the joint independence model represented by the graph $\mathcal{G}$ in Example \ref{ex:join_indep} on $2 \times 3 \times 4$ contingency tables corresponding to random variables $X_1$, $X_2$, and $X_3$.  We get a matrix $M_{C}$ for each clique $C$ of $\mathcal{G}$:
    \[M_{\{X_1\}}=\begin{pmatrix}
        p_{111} & p_{112} & p_{113} & p_{114} & p_{121} & p_{122} & p_{123} & p_{124} & p_{131} & p_{132} & p_{133} & p_{134} & u_{1++}\\
        p_{211} & p_{212} & p_{213} & p_{214} & p_{221} & p_{222} & p_{223} & p_{224} & p_{231} & p_{232} & p_{233} & p_{234} & u_{2++}\\
    \end{pmatrix}\]
    \[M_{\{X_2,X_3\}}=\begin{pmatrix}
        p_{111} &p_{211} & u_{+11} \\
        p_{112} &p_{212} & u_{+12}\\
        p_{113} &p_{213} & u_{+13}\\
        p_{114} &p_{214} & u_{+14}\\
        p_{121} &p_{221} & u_{+21}\\
        p_{122} &p_{222} & u_{+22}\\
        p_{123} &p_{223} & u_{+23}\\
        p_{124} &p_{224} & u_{+24}\\
        p_{131} &p_{231} & u_{+31}\\
        p_{132} &p_{232} & u_{+32}\\
        p_{133} &p_{233} & u_{+33}\\
        p_{134} &p_{234} & u_{+34}\\
    \end{pmatrix}.\]
    The likelihood ideal for the joint independence model of $\mathcal{G}$ is generated by the $2\times 2$ minors of these matrices.
\end{example}

\begin{example}\label{ex:HardyWeinberg2}
   We revisit the Hardy--Weinberg curve from Example~\ref{ex:HardyWeinberg} using  Theorem ~\ref{thm:LCtoric}. Up to coordinate rescaling, as a toric model, it is represented by the following matrix:
    \[ A = \begin{pmatrix}
        0 & 1 & 2 \\
        2 & 1 & 0\\
    \end{pmatrix}. \]
    In this case, the product $AM$ is 
    \[AM = \begin{pmatrix}
        0 & 1 & 2 \\
        2 & 1 & 0\\
    \end{pmatrix}\begin{pmatrix}
        u_0 & p_0\\
        u_1 & p_1\\
        u_2 & p_2
    \end{pmatrix} = \begin{pmatrix}
        u_1+2u_2 & p_1+2p_2\\
        2u_0+u_1 & 2p_0+p_1\\
    \end{pmatrix}.\]
    By Theorem \ref{thm:LCtoric}, the ideal 
    \[\langle 4p_0p_2 - p_1^2, (u_1+2u_2)(2p_0+p_1)-(2u_0+u_1)(p_1+2p_2) \rangle :\langle p_0\cdots p_np_+\rangle^\infty\]
    is the likelihood ideal.
\end{example}

\begin{example}
    Consider the 3-chain undirected graphical model $\mathcal{G}$ on binary random variables $X_1$, $X_2$, and $X_3$ as in Example \ref{ex:3chain}. This represents the conditional independence of $X_1$ and $X_3$ given $X_2$, i.e. $X_1 \indep X_3 \, |\, X_2$. Using Theorem \ref{thm:LCtoric}, we can compute that the likelihood ideal has 19 quadric generators and has degree 20. We can attempt to apply the same method from Theorem \ref{thm:LCindependence} directly as in Example \ref{ex:join_dep2} by computing matrices for each clique $\{X_1,X_2\}$ and $\{X_2,X_3\}$ of $\mathcal{G}$:
    \[M_{\{X_1,X_2\}} = \begin{pmatrix}
        p_{111} & p_{112} & u_{11+}\\
        p_{121} & p_{122} & u_{12+}\\
        p_{211} & p_{212} & u_{21+}\\
        p_{221} & p_{222} & u_{22+}
    \end{pmatrix} \hspace{0.2in} M_{\{X_2,X_3\}} = \begin{pmatrix}
        p_{111} & p_{211} & u_{+11}\\
        p_{112} & p_{212} & u_{+12}\\
        p_{121} & p_{221} & u_{+21}\\
        p_{122} & p_{222} & u_{+22}
    \end{pmatrix}. \]

    \noindent Of course, the $2\times 2$ minors of these is \textit{not} the likelihood ideal --  the ideal of $2\times 2$ minors of the first two columns of each matrix do not even cut out the model. For conditional independence models like this, however, we have a fix: After fixing a value for $X_2$, $X_1$ and $X_3$ are completely independent. This is expressed by the $2\times 2$ minors of the matrices we get by spliting up each matrix above into two cases according to the outcome of the central vertex  $X_2$:

    \[M_{\{X_1,1\}}=\begin{pmatrix}
        p_{111}&p_{112}&u_{11+}\\
      p_{211}&p_{212}&u_{21+}\\
      \end{pmatrix}\hspace{0.2in}
      M_{\{X_1,2\}}=\begin{pmatrix}
      p_{121}&p_{122}&u_{12+}\\
      p_{221}&p_{222}&u_{22+}
    \end{pmatrix}\]
    \hspace{13.8cm} .
    \[M_{\{X_3,1\}}=\begin{pmatrix}
        p_{111}&p_{211}&u_{+11}\\
      p_{112}&p_{212}&u_{+12}\\
      \end{pmatrix}\hspace{0.2in}
      M_{\{X_3,2\}}=\begin{pmatrix}
      p_{121}&p_{221}&u_{+21}\\
      p_{122}&p_{222}&u_{+22}
    \end{pmatrix}\] 

\noindent  
           This is nearly the likelihood ideal, although it misses the following quadrics which fill in the relations between the cases $X_1 = 1$ and $X_2 = 2$:
           \[ p_{21+}u_{22+} - p_{22+}u_{21+},\hspace{0.1in}  p_{21+}u_{+22} - p_{+22}u_{21+}, \hspace{0.1in}  p_{21+}u_{+2+} - p_{+2+}u_{21+} , \] \[p_{+12}u_{22+} - p_{22+}u_{+12},\hspace{0.1in}  p_{+12}u_{+22} - p_{+22}u_{+12},\hspace{0.1in}  p_{+12}u_{+2+} - p_{+2+}u_{+12} , \]\[p_{+1+}u_{22+} - p_{22+}u_{+1+},\hspace{0.1in}  p_{+1+}u_{+22} - p_{+22}u_{+1+}, \hspace{0.1in} p_{+1+}u_{+2+} - p_{+2+}u_{+1+}.  \]
            \noindent These quadrics are exactly the generators of $I_2(AM)$ (from Proposition~\ref{thm:LCtoric}) that are not already contained in the ideal $I_2(M_{\{X_1,1\}}) + I_2(M_{\{X_1,2\}}) + I_2(M_{\{X_3,1\}}) + I_2(M_{\{X_3,2\}})$.
\end{example}

For more general log--linear models, there does not appear to be a similar way to reconcile the methods in Theorem \ref{thm:LCindependence} and Corollary \ref{cor:LCjointindependence}.

\begin{example}
    Consider a $3$-way contingency table consisting of binary random variables $X_1,X_2,$ and $X_3$. The \textit{no $3$-way interaction model} is a log--linear (but not undirected graphical) model given by the generators $\{\{X_1,X_2\},\{X_1,X_3\},\{X_2,X_3\}\}$ and $A$ matrix:
\begin{equation*}
        A=\begin{pmatrix}
        1 & 1&0&0&0&0&0&0 \\
        0 & 0&1&1&0&0&0&0 \\
        0 & 0&0&0&1&1&0&0 \\
        0 & 0&0&0&0&0&1&1\\
        1 & 0&1&0&0&0&0&0\\
        0 & 1&0&1&0&0&0&0\\
        0 & 0&0&0&1&0&1&0\\
        0 & 0&0&0&0&1&0&1\\
        1 & 0&0&0&1&0&0&0\\
        0 & 1&0&0&0&1&0&0\\
        0 & 0&1&0&0&0&1&0\\
        0 & 0&0&1&0&0&0&1
    \end{pmatrix}.
    \end{equation*}
    If one attempts to break up the matrices as in Example \ref{ex:3chain}, the $2 \times 2$ minors produce an ideal that contains $I_A = \langle p_{111}p_{122}p_{212}p_{221} - p_{112}p_{121}p_{211}p_{222} \rangle$, but is neither contained in nor contains the likelihood ideal. 
    
    The ideal $I_A+I_2(AM)$ from Theorem \ref{thm:LCtoric} is contained in the likelihood ideal but is not prime. Saturating by $\langle p_0\dots p_np_+\rangle^\infty$ fixes this by introducing a single quartic:
    \begin{multline*}
        p_{222}^{3}u_{11+} + p_{222}^2(p_{112}(u_{121}-u_{222})-p_{111}u_{+22})\\
        + p_{222}(p_{112}p_{121}(u_{211}-u_{222}) + p_{111}p_{122}(u_{2+2})) - p_{111}p_{122}p_{212}u_{22+}. \qedhere
    \end{multline*}
\end{example}
\subsection*{Acknowledgement}
Thanks to Daniel Erman and Rudy Yoshida for invaluable discussions and to Bernd Sturmfels for his mentorship during this project. Thanks also to Ivan Aidun, Lisa Seccia, Patricia Klein, and Serkan Hoşten for their thoughts. Part of this research was performed while the authors were visiting the Institute for Mathematical and Statistical Innovation (IMSI), which is supported by the National Science Foundation (Grant No. DMS-1929348). Some of the computations used for this research were carried out on the Texas A\&M University Department of Mathematics Whistler cluster, access to which is supported by the National Science Foundation (Grant No. DMS-2201005). Cobb acknowledges the support of the National Science Foundation Grant DMS-2402199.  Faust is supported by DMS-2052519.

\bibliography{refs}{}

\end{document}